\documentclass[reqno, 11pt]{amsart}
\usepackage{amsmath}
 \usepackage{amssymb}
\usepackage{amsthm}
\usepackage{times}
\usepackage{latexsym}
\usepackage[mathscr]{eucal}

\numberwithin{equation}{section}
 
  \newtheorem{theorem}{Theorem}[section]
  \newtheorem{proposition}[theorem]{Proposition}
  \newtheorem{lemma}[theorem]{Lemma}
  \newtheorem{corollary}[theorem]{Corollary}

  \newtheorem{definition}[theorem]{Definition}
  \newtheorem{example}[theorem]{Example}

\title[Quasi Generalized CR-lightlike Submanifolds]{Quasi Generalized CR-lightlike submanifolds of indefinite nearly Sasakian manifolds}

\author[Fortun\'{e} Massamba, Samuel Ssekajja ]{Fortun\'{e} Massamba*, Samuel Ssekajja**}

\newcommand{\acr}{\newline\indent}

\address{\llap{*\,} School of Mathematics, Statistics and Computer Science\acr
 University of KwaZulu-Natal\acr
 Private Bag X01, Scottsville 3209\acr
South Africa}
\email{massfort@yahoo.fr, Massamba@ukzn.ac.za} 
\thanks{}

\address{\llap{**\,} School of Mathematics, Statistics and Computer Science\acr
 University of KwaZulu-Natal\acr
 Private Bag X01, Scottsville 3209\acr
South Africa}
\email{ssekajja.samuel.buwaga@aims-senegal.org} 
\thanks{}



\begin{document}
 
\begin{abstract} 
In this paper, we introduce and study a new class of CR-lightlike submanifold of an indefinite nearly Sasakian manifold, called Quasi Generalized Cauchy-Riemann (QGCR) lightlike submanifold. We give some characterization theorems for the existence of QGCR-lightlike submanifolds and finally derive necessary and sufficient conditions for  some distributions to be integrable. 
\end{abstract}

\maketitle

\section{Introduction}   
 
Blair \cite{bl2} introduced the notions of a nearly Sasakian manifold as a special class of almost contact metric manifolds. An indefinite nearly Sasakian manifold differs from  an indefinite Sasakian manifold, since in the former, the manifold is not necessarily normal \cite{bl1}. In fact, any normal nearly Sasakian manifold is Sasakian  (see \cite{bl2} and references therein for more details). From then, many papers have appeared on these manifolds and their submanifolds \cite{bl1}, \cite{be}, \cite{mo} and \cite{ta}. In these papers, the geometry is restricted to a Riemannian case and thus, little or no attempt has been made to investigate their lightlike (null) cases. Lightlike geometry has its applications in mathematical physics, in particular, general relativity and electromagnetism \cite{db}. Differential geometry of lightlike submanifolds was introduced by Bejancu and Duggal \cite{db} and later studied by many authors  \cite{dj}, \cite{ds3}, \cite{ds4}, \cite{ma1}, \cite{ma2}, \cite{ma3} and references therein.

In \cite{ds3}, the authors initiated the study of Generalized Cauchy-Riemann (GCR) lightlike submanifolds of an indefinite Sasakian manifold, in which the structural vector field, $\xi$, of the almost contact metric structure $(\overline{\phi},\eta,\xi,\overline{g})$ was assumed to be tangent to the submanifold. Moreover, when $\xi$ is tangent to the submanifold, Calin \cite{ca} proved that it belongs to the screen distribution. However, the structural vector field is globally defined on the tangent bundle of the ambient manifold, which implies that other classes of submanifolds with non-tangential structural vector fields are certainly possible. Recently, a few  papers have been published on the subject, focusing on ascreen and generic lightlike submanifolds (\cite{ds4}, \cite{ji}). Thus, the absence of evidence of research in the geometry of lightlike submanifolds of nearly Sasakian manifolds and the fact that $\xi$ belongs to the tangent bundle of the ambient space has motivated us to introduce a new class of CR-lightlike submanifold of a nearly Sasakian manifold, known as Quasi Generalized Cauchy-Riemann (QGCR) lightlike submanifold.

The objective of this paper is to characterize totally umbilical and totally geodesic QGCR-lightlike submanifolds of a nearly Sasakian manifold. The rest of the paper is organized as follows. In Section \ref{Preli}, we present the basic notions of nearly Sasakian structures and lightlike submanifolds which we refer to in the remaining sections. In Section \ref{AlmostGe}, we introduce QGCR-lightlike submanifolds. Section \ref{existence} is devoted to the non-existance theorms regarding totally umbilical and totally geodesic QGCR-lightlike submanifolds. Finally, in Section \ref{Integra} we derive the necessarily and sufficient conditions for the integrability of the key distributions of a QGCR-lightlike submanifold of an indefinite nearly Sasakian manifold.

\section{Preliminaries}\label{Preli}

Let $\overline{M}$ be a $(2n + 1)$-dimensional manifold endowed with an almost contact structure $(\overline{\phi}, \xi, \eta)$, i.e. $\overline{\phi}$ is a tensor field of type $(1, 1)$, $\xi$ is a vector field, and $\eta$ is a 1-form satisfying
\begin{equation}\label{equa1}
\overline{\phi}^{2} = -\mathbb{I} + \eta
\otimes\xi,\;\;\eta(\xi)= 1 ,\;\;\eta\circ\overline{\phi} =
0 \;\;\mbox{and}\;\;\overline{\phi}(\xi) = 0.
\end{equation}
Then $(\overline{\phi}, \xi, \eta,\,\overline{g})$ is called an indefinite almost contact metric structure on $\overline{M}$ if $(\overline{\phi}, \xi, \eta)$ is an almost contact structure on $\overline{M}$ and $\overline{g}$ is a semi-Riemannian metric on $\overline{M}$ such that \cite{bcgh}, for any vector field $\overline{X}$, $\overline{Y}$ on  $\overline{M}$,
\begin{equation}\label{equa2}
 \overline{g}(\overline{\phi}\,\overline{X}, \overline{\phi}\,\overline{Y}) = \overline{g}(\overline{X}, \overline{Y}) -  \eta(\overline{X})\,\eta(\overline{Y}),
\end{equation}
It follows that, for any vector $\overline{X}$ on $\overline{M}$,
 \begin{equation}\label{equa3}
  \eta(\overline{X}) =  \overline{g}(\xi,\overline{X}).
 \end{equation}

If, moreover, 
\begin{equation}\label{eqz}
         (\overline{\nabla}_{\overline{X}} \overline{\phi})\overline{Y}+(\overline{\nabla}_{\overline{Y}} \overline{\phi})\overline{X}=2\overline{g}(\overline{X}, \overline{Y})\xi-\eta(\overline{Y})\overline{X}-\eta(\overline{X})\overline{Y},
    \end{equation}
for any vector fields $\overline{X}$, $\overline{Y}$ on $\overline{M}$, where $\overline{\nabla}$ is the Levi-Civita connection for the semi-Riemannian metric $\overline{g}$, we call $\overline{M}$ an \textit{indefinite nearly Sasakian manifold}.   

We denote by $\Gamma(\Xi)$ the set of smooth sections of the vector bundle $\Xi$. Let $\Omega$ be the fundamental 2-form of $\overline{M}$ defined by
\begin{equation}
 \Omega(\overline{X}, \overline{Y}) = \overline{g}(\overline{X}, \overline{\phi}\,\overline{Y}),\;\;\overline{X},\;\overline{Y}\in\Gamma(T \overline{M}).
\end{equation}
Replacing $\overline{Y}$ by $\xi$ in (\ref{eqz}) we obtain
   \begin{equation}
    \overline{\nabla}_{\overline{X}} \xi-\overline{\phi}^2\overline{\nabla}_\xi \overline{X}+\overline{\phi}\,\overline{\nabla}_\xi \overline{\phi}\,\overline{X} =-\overline{\phi}\,\overline{X},\label{eq10}
   \end{equation}
for any $\overline{X}\in \Gamma (T\overline{M})$.

Introduce a (1,1)-tensor $\overline{H}$ on $\overline{M}$ taking
\begin{equation*}
 \overline{H}\,\overline{X}=-\overline{\phi}^2\overline{\nabla}_\xi \overline{X} + \overline{\phi}\,\overline{\nabla}_\xi \overline{\phi}\,\overline{X}, 
\end{equation*}
for any $\overline{X}\in \Gamma (T\overline{M})$, such that (\ref{eq10}) reduces to 
   \begin{equation}\label{v10}
    \overline{\nabla}_{\overline{X}} \xi+\overline{H}\,\overline{X} =-\overline{\phi}\, \overline{X}.
   \end{equation}
\begin{lemma}
 The linear operator $\overline{H}$ has the properties
 \begin{align}
  & \overline{H}\,\overline{\phi} + \overline{\phi}\,\overline{H}=0,\;\;\overline{H}\xi=0,\;\;\eta\circ \overline{H}=0, \nonumber\\
  \mbox{and}\;\;& \overline{g}(\overline{H}\,\overline{X}, \overline{Y})=-\overline{g}(\overline{X}, \overline{H}\,\overline{Y})\;\;\;\; (\mbox{i.e.}\;\; \overline{H} \;\;\mbox{is skew-symmetric}).
 \end{align}
 \end{lemma} 
 \begin{proof}
  The proof follows from a straightforward calculation.
 \end{proof}
 By (\ref{v10}), it is easy to check that 
\begin{equation}\label{p}
 \overline{\nabla}_\xi \xi=0.
\end{equation} 
 The fundamental 2-form $\Omega$ and the 1-form $\eta$ are related as follows.
\begin{lemma}
Let $(\overline{M}, \overline{\phi}, \xi, \eta,\,\overline{g})$ be an indefinite nearly Sasakian manifold. Then,  
 \begin{equation}\label{RelaOmegaETa}
  \Omega(\overline{X}, \overline{Y}) = d\eta(\overline{X}, \overline{Y}) + \overline{g}(\overline{H}\, \overline{X}, \overline{Y}),
 \end{equation}
 for any $\overline{X}$, $\overline{Y}\in\Gamma(T \overline{M})$. Moreover, $\overline{M}$ is Sasakian if and only if $\overline{H}$ vanishes identically on $\overline{M}$.
\end{lemma}
 \begin{proof}
  The relation (\ref{RelaOmegaETa}) follows from a straightforward calculation. The second assertion follows from Theorem 3.2 in \cite{bl1}.
 \end{proof}
Note that, for any $\overline{X}$, $\overline{Y}$, $\overline{Z}\in\Gamma(T \overline{M})$,
\begin{equation}
 \overline{g}((\overline{\nabla}_{\overline{Z}}\overline{\phi})\overline{X}, \overline{Y}) = -  \overline{g}( \overline{X}, (\overline{\nabla}_{\overline{Z}}\overline{\phi})\overline{Y}).
\end{equation}
This means that the tensor $\overline{\nabla}\,\overline{\phi}$ is skew-symmetric.
 
  Let $(\overline{M},\overline{g})$ be an $(m + n)$-dimensional semi-Riemannian manifold of constant index $\nu$, $1\le \nu\le m+n$ and $M$ be a submanifold of $\overline{M}$ of codimension $n$. We assume that both $m$ and $n$ are $\ge 1$. At a point $p\in M$, we define the orthogonal complement $T_{p} M^{\perp}$ of the tangent space $T_{p} M$ by
 $$
 T_{p} M^{\perp} = \{X\in\Gamma(T_{p} M): \overline{g}(X, Y)=0,\; \forall \, Y\in\Gamma(T_{p} M)\}.
 $$
 We put $\mathrm{Rad} \, T_{p} M = \mathrm{Rad}\, T_{p} M^{\perp} = T_{p} M \cap T_{p} M^{\perp}$. The submanifold $M$ of $\overline{M}$ is said to be $r$-lightlike submanifold (one supposes that the index of $\overline{M}$ is $\nu \ge r$), if the mapping 
 $$
 \mathrm{Rad} \, T M: p\in M \longrightarrow\mathrm{Rad}\, T_{p} M 
 $$
 defines a smooth distribution on $M$ of rank $r > 0$. We call $\mathrm{Rad}\,T M$ the radical distribution on $M$. In the sequel, an $r$-lightlike submanifold will simply be called a \textit{lightlike submanifold} and $g$ is \textit{lightlike metric}, unless we need to specify $r$.
 
 Let $S(T M)$ be a screen distribution which is a semi-Riemannian complementary distribution of $\mathrm{Rad}\,T M$ in $T M$, that is,
 \begin{equation}\label{eq05}
  T M = \mathrm{Rad}\,T M \perp S(T M).
 \end{equation} 
Choose a screen transversal bundle $S(TM^\perp)$, which is semi-Riemannian and complementary to $\mathrm{Rad}\, TM$ in $TM^\perp$. Since, for any local basis $\{E_i \}$ of  $\mathrm{Rad}\,TM$, there exists a local null frame $\{N_i\}$ of sections with values in the orthogonal complement of $S(T M^\perp)$ in $S(T M )^\perp$  such that $g(E_i , N_j ) = \delta_{ij}$, it follows that there exists a lightlike transversal vector bundle $l\mathrm{tr}(TM)$ locally spanned by $\{N_i\}$ \cite{db}. 

Let $\mathrm{tr}(TM)$ be complementary (but not orthogonal) vector bundle to $TM$ in $T\overline{M}$. Then, 
\begin{align}\label{eq04}
           &\mathrm{tr}(TM)=l\mathrm{tr}(TM)\perp S(TM^\perp),\\\label{eq08}
  T\overline{M}= & S(TM)\perp S(TM^\perp)\perp\{\mathrm{Rad}\, TM\oplus l\mathrm{tr}(TM)\} .
\end{align}
Note that the distribution $S(TM)$ is not unique, and is canonically isomorphic to the factor vector bundle $TM/ \mathrm{Rad}\, TM$. \cite{ds3} 
 
We say that a lightlike submanifold $M$ of $\overline{M}$ is 
\begin{enumerate}
 \item $r$-lightlike if $1\leq r< min\{m,n\}$;
 \item co-isotropic if $1\leq r=n<m$,  $S(TM^\perp)=\{0\}$;
 \item isotropic if $1\leq r=m<n$,  $S(TM)=\{0\}$;
 \item totally lightlike if $r=n=m$,  $S(TM)=S(TM^\perp)=\{0\}$.
\end{enumerate}
Similar to \cite{ds4}, we use the following range of indices in this paper,
\begin{equation*}
 i,j,k\in\{1,\ldots, r\}, \hspace{.2cm}\alpha,\beta,\gamma\in\{r+1,\ldots, n\}.
\end{equation*}

Consider a local quasi-orthonormal fields of frames of $\overline{M}$ along $M$, on $U$ as 
\begin{equation*}
\{ E_1,\cdots, E_r,N_1,\cdots, N_r,X_{r+1},\cdots,X_{m},W_{1+r},\cdots, W_{n}\},
\end{equation*}
where
 $\{X_{r+1},\cdots,X_{m}\}$ and $\{W_{1+r},\ldots, W_n\}$ are respectively othornomal bases of $\Gamma(S(TM)|_{U})$ and $\Gamma(S(TM^{\perp})|_{U})$  and that $\epsilon_\alpha=\overline{g}(W_\alpha,W_\alpha)$ be the signatures of $W_\alpha$.
 
Let $P$ be the projection morphism of $TM$ on to $S(TM)$. Using the decomposition (\ref{eq08}), consider the projection morphisms $L$ and $S$ of
$\mathrm{tr}(T M )$ on $l\mathrm{tr}(T M)$ and $S(T M^{\perp})$, respectively. Then,  the Gauss-Wiengartein equations \cite{ds2} of an $r$-lightlike submanifold $M$ and $S(TM)$ are the following;
  \begin{align}
     & \overline{\nabla}_X Y=\nabla_X Y+\sum_{i=1}^r h_i^l(X,Y)N_i+\sum_{\alpha=r+1}^n h_\alpha^s(X,Y)W_\alpha,\label{eq11}\\
     & \overline{\nabla}_X N_i=-A_{N_i} X+\sum_{j=1}^r \tau_{ij}(X) N_j+\sum_{\alpha=r+1}^n \rho_{i\alpha}(X)W_\alpha,\label{eq31}\\
     & \overline{\nabla}_X W_\alpha=-A_{W_\alpha} X+\sum_{i=1}^r \varphi_{\alpha i}(X) N_i+\sum_{\beta=r+1}^n \sigma_{\alpha\beta}(X)W_\beta,\label{eq32}\\
     & \nabla_X P Y=\nabla_X^* PY+\sum_{i=1}^r h_i^*(X, P Y)E_i,\\
     & \nabla_X E_i=-A_{E_i}^* X-\sum_{j=1}^r \tau_{ji}(X) E_j,\;\;\;\; \forall\; X,Y\in \Gamma(TM)\label{eq50},
  \end{align}
where $h^{l}(X, Y)= L h(X, Y)$, $h^{s}(X, Y)= S h(X, Y)$,  $\nabla$ and $\nabla^*$ are the induced connections on $TM$ and $S(TM)$ respectively, $h_i^l$ and $h_\alpha^s$ are symmetric bilinear forms known as \textit{local lightlike} and \textit{screen fundamental} forms of $TM$ respectively. Also $h_i^*$ are the \textit{ second fundamental forms} of $S(TM)$. $A_{N_i}$, $A_{E_i}^*$ and $A_{W_\alpha}$ are linear operators on $TM$ while $\tau_{ij}$, $\rho_{i\alpha}$, $\varphi_{\alpha i}$ and $\sigma_{\alpha\beta}$ are 1-forms on $TM$. It is known \cite{db, ds2} that 

\begin{equation}\label{13}
 h_i^l(X,Y)= \overline{g}(\overline{\nabla}_X Y,E_i),\hspace{.3cm} \forall X,Y\in \Gamma(TM),
\end{equation}
from which we deduce the independence of $h_i^l$ on  the choice of $S(TM)$. 

The \textit{second fundamental tensor} of $M$ is given by
 \begin{equation}\label{h1}
 h(X,Y)=\sum_{i=1}^r h_i^l(X,Y)N_i+\sum_{\alpha=r+1}^n h_\alpha^s(X,Y)W_\alpha,
\end{equation}
for any $X,Y\in \Gamma(TM)$.  It is easy to see that  $\nabla^*$ is a metric connection on $S(TM)$ while $\nabla$ is generally not a metric connection and is given by
      \begin{equation*}
         (\nabla_X g)(Y,Z)=\sum_{i=1}^r\{h_i^l(X,Y)\lambda_i(Z)+h_i^l(X,Z)\lambda_i(Y)\},
      \end{equation*}
for any $X,Y\in \Gamma(TM)$ and $\lambda_i$ are 1-forms given by
      \begin{equation}\label{eqc}
          \lambda_i(X)=\overline{g}(X,N_i), \hspace{.2cm} \forall X\in \Gamma(TM).
      \end{equation}
The above three local second fundamental forms are related to their shape operators by the following set of equations
      \begin{align}
         & g(A_{E_i}^*X,Y)=h_i^l(X,Y)+\sum_{j=1}^rh_j^l(X,E_i)\lambda_j(Y),\hspace{.1cm}\bar{g}(A_{E_i}^*X,N_j)=0,\label{eqe}\\
         & g(A_{W_\alpha}X,Y)=\epsilon_\alpha h_\alpha^s(X,Y)+\sum_{i=1}^r \varphi_{\alpha i}(X)\lambda_i(Y),\label{eq07}\\
         & \bar{g}(A_{W_\alpha}X,N_i)=\epsilon_\alpha \rho_{i\alpha}(X),\\
         & g(A_{N_i}X,Y)=h_i^*(X,\mathcal{P}Y),\;\;\; \lambda_j(A_{N_i}X)+\lambda_i(A_{N_j}X)=0\label{k},
      \end{align}
for any $X,Y\in \Gamma(TM)$. 

For any $r$-lightlike submanifold, replacing $Y$ by $E_i$ in (\ref{eq07}), we get
         \begin{equation}\label{eqg}
             h_\alpha^s(X,E_i)=-\epsilon_\alpha \varphi_{\alpha i}(X), \hspace{.3cm} \forall X\in \Gamma(TM).
         \end{equation}
         
 A lightlike submanifold $(M,g)$, of a semi-Riemannian manifold $(\overline{M},\overline{g})$ is said to be totally umbilical in $\overline{M}$ \cite{ds2} if there is a smooth transversal vector field $\mathcal{H}\in \Gamma(\mathrm{tr}(TM))$, called the transversal curvature vector of $M$ such that 
          \begin{align}\label{eq17}
             h(X,Y)=\mathcal{H}g(X,Y),
          \end{align}
for all  $X,Y\in \Gamma(TM)$. 

Moreover, it is easy to see that $M$ is totally umbilical in $\overline{M}$, if and only if on each coordinate neighborhood $\mathscr{U}$ there exist  smooth vector fields $\mathscr{H}^l\in\Gamma(l\mathrm{tr}(TM))$ and $\mathscr{H}^s\in\Gamma(S(TM^\perp))$ and smooth functions $\mathscr{H}_i^l\in F(l\mathrm{tr}(TM))$ and $\mathscr{H}_\alpha^s\in F(S(TM^\perp))$ such that  
  \begin{align}
   h^l(X,Y) & =\mathscr{H}^l g(X,Y),\;\;\; h^s(X,Y)=\mathscr{H}^s g(X,Y),\nonumber\\
   h_i^l(X,Y) & =\mathscr{H}_i^l g(X,Y),\;\;\;h_\alpha^s(X,Y)=\mathscr{H}_\alpha^s g(X,Y),
  \end{align} 
for all $X,Y\in\Gamma(TM)$. 

The above definition is independent of the choice of the screen distribution.

\section{Quasi Generalized CR-lightlike Submanifolds}\label{AlmostGe}

Generally, the structure vector field $\xi$ belongs to $T\overline{M}$. Therefore, we define it according to decomposition  (\ref{eq08}) as follows;
         \begin{equation}\label{eq2}
              \xi=\xi_S+\xi_{S^\perp}+\xi_R+\xi_l,
         \end{equation}
where $\xi_S$ is a smooth vector field of $S(TM)$ and $\xi_{S^\perp}$, $\xi_R$, $\xi_l$ are  defined as follows

          \begin{equation}\label{eq0}
              \xi_{R}=\sum_{i=1}^ra_i E_i, \;\; \xi_{l}=\sum_{i=1}^rb_i N_i,  \;\; 
              \xi_{S^\perp}=\sum_{\alpha=r+1}^nc_\alpha W_\alpha
          \end{equation}
with $a_i=\eta(N_i)$, $b_i=\eta(E_i)$ and  $c_\alpha=\epsilon_\alpha\eta(W_\alpha)$ all smooth functions on $\overline{M}$.

Generalised Cauchy Riemann (GCR) lightlike submanifolds were introduced in \cite{ds2, ds3}, in which the structure vector field $\xi$ was assumed tangent to the submanifold. Contrary to this assumption, we introduce a special class of $CR$-lightlike submanifold in which $\xi$ belongs to $T\overline{M}$, called  \textit{Quasi Generalized Cauchy-Riemann (QGCR)-lightlike submanifold} as follows.
\begin{definition}\label{def2}{\rm
Let $(M,g,S(TM),S(TM^\perp))$  be a lightlike submanifold of an indefinite nearly Sasakian manifold $(\overline{M}, \overline{g}, \overline{\phi},\xi,\eta)$. We say that $M$ is Quasi Generalized Cauchy-Riemann (QGCR)-lightlike submanifold of $\overline{M}$ if the following conditions are satisfied:
\begin{enumerate}
 \item [(i)] there exist two distributions $D_1$ and $D_2$ of $Rad(TM)$  such that 
        \begin{equation}\label{eq03}
             \mathrm{Rad}\, TM = D_1\oplus D_2, \;\;\overline{\phi} D_1=D_1, \;\;\overline{\phi} D_2\subset S(TM),
        \end{equation}
 \item [(ii)] there exist vector bundles $D_0$ and $\overline{D}$ over $S(TM)$ such that 
         \begin{align} 
              & S(TM)=\{\overline{\phi} D_2 \oplus \overline{D}\}\perp D_0,\\ 
              \mbox{with}\;\;\; &\overline{\phi}D_{0} \subseteq D_{0},\;\;   \overline{D}= \overline{\phi} \, \mathcal{S}\oplus \overline{\phi} \,\mathcal{L}, \label{s81}
         \end{align}    
\end{enumerate}
where $D_0$ is a non-degenerate distribution on $M$, $\mathcal{L}$ and $\mathcal{S}$ are respectively vector subbundles of $l\mathrm{tr}(TM)$  and $S(TM^{\perp})$.
}
\end{definition}

If $D_{1}\neq \{0\}$, $D_0\neq \{0\}$, $D_2\neq \{0\}$ and $\mathcal{S}\neq \{0\}$, then $M$ is called a \textit{proper QGCR-lightlike submanifold}.

Let $M$ be a QGCR-lightlike submanifold of an indefinite nearly Sasakian manifold $\overline{M}$. If the structure vector field $\xi$ is tangent to $M$, then, $\xi\in\Gamma(S(T M))$. The proof of this is similar to one given by Calin in Sasakian case \cite{ca}. In this case,  if $X\in\Gamma(\mathcal{S})$ and $Y\in\Gamma(\mathcal{L})$, then $\eta(X)=\eta(Y)=0$ and
$$
\overline{g}(\overline{\phi}X, \overline{\phi}Y) = \overline{g}(X, Y) - \eta(X)\eta(Y) =0,
$$
which reduces the direct sum  $ \overline{D}$ in (\ref{s81}) to the orthogonal direct sum $ \overline{D}= \overline{\phi} \,\mathcal{S}\perp \overline{\phi} \,\mathcal{L}$, and thus $ \overline{\phi} \,\overline{D}= \mathcal{S}\perp  \mathcal{L}$.  Since $\xi\in\Gamma(S(T M))$ and $\xi$ is neither a vector field in  $\overline{\phi} D_2$ nor in $\overline{D}$, $\xi$ is in $D_{0}$. By $\overline{\phi}D_{0} \subseteq D_{0}$, there exist a distribution $D_{0}'$ of rank $(rank(D_{0})-1)$ and satisfying $\overline{\phi}D_{0}' = D_{0}'$ such that $D_{0}=D_{0}'\perp\langle\xi\rangle$, where $\langle\xi\rangle$ is the 1-dimensional distribution spanned by $\xi$. Therefore, the QGCR-lightlike submanifold tangent to $\xi$ reverts to a GCR-lightlike submanifold \cite{ds3}. 

\begin{proposition}
 A QGCR-lightlike submanifold $M$ of an indefinite nearly Sasakian manifold $\overline{M}$ tangent to the structure vector field $\xi$ is a GCR-lightlike submanifold.
 \end{proposition}  
 Next, we follow Yano-Kon \cite[p. 353]{yk} definition of contact CR-submanifolds and state the following definition for a quasi contact CR-lightlike submanifold.
\begin{definition}\label{def3}{\rm
Let $(M,g,S(TM),S(TM^\perp))$  be a lightlike submanifold of an indefinite nearly Sasakian manifold $(\overline{M}, \overline{g}, \overline{\phi},\xi,\eta)$. We say that $M$ is quasi contact CR-lightlike submanifold of $\overline{M}$ if the following conditions are satisfied;
\begin{enumerate}
 \item [(i)] $\mathrm{Rad}\,TM$ is a distribution on M such that $\mathrm{Rad}\,TM\cap\overline{\phi}(\mathrm{Rad}\,TM)=\{0\}$; 
 \item [(ii)] there exist vector bundles $D_0$ and $D'$ over $S(TM)$ such that 
         \begin{align} 
              & S(TM)=\{\overline{\phi}(\mathrm{Rad}\,TM)\oplus D'\}\perp D_0,\\ 
              \mbox{with}\;\;\; &\overline{\phi}D_{0} \subseteq D_{0},\;\; D'=\overline{\phi} L_{1} \oplus \overline{\phi}l\mathrm{tr} (T M), \label{s888}
         \end{align}    
\end{enumerate}
where $D_0$ is a non-degenerate, $L_{1}$ is vector subbundle of $S(TM^{\perp})$.
}
\end{definition}
It is easy to see that when the structure vector field $\xi$ is tangent to the quasi Contact CR-lightlike submanifold $M$, then $M$ is a contact CR.
\begin{proposition}
 A QGCR-lightlike submanifold of an indefinite nearly Sasakian manifold $\overline{M}$, is a quasi contact CR if and only if $D_{1} = \{0\}$.
\end{proposition}
\begin{proof}
 Let $M$ be a quasi contact CR-lightlike submanifold. Then $\overline{\phi}(\mathrm{Rad}\,TM)$ is a distribution on $M$ such that $ \overline{\phi}(\mathrm{Rad}\,TM)\cap\mathrm{Rad}\,TM=\{0\}$. Therefore, $D_{2} = \mathrm{Rad}\,TM$ and $D_{1}=\{0\}$. Hence, $ \overline{\phi}(l\mathrm{tr} (T M)) \cap l\mathrm{tr} (T M)=\{0\}$. Then it follows that $\overline{\phi}(l\mathrm{tr} (T M))\subset S(T M)$. The converse is obvious.
\end{proof}
From (\ref{eq05}), the tangent bundle of any QGCR lightlike submanifold, $TM$, can be rewritten as
   \begin{equation*}
       TM=D \oplus \widehat{D},\;\;\mbox{where}\;D=D_0\perp D_{1}\;\mbox{and}\;\widehat{D}=\{D_{2}\perp\overline{\phi} D_{2}\} \oplus\overline{D}.
   \end{equation*}
   
Notice that $D$ is invariant with respect to $\overline{\phi}$ while $\widehat{D}$ is not generally anti-invariant with respect to $\overline{\phi}$.

Note the following for a proper QGCR-lightlike submanifold $(M,g,S(TM)$, $S(TM^\perp))$ of an indefinite almost contact metric manifolds $\overline{M}$ according to Definition \ref{def2}:
          \begin{enumerate}
               \item Condition (i) implies that $\dim(\mathrm{Rad}\, TM)=s\ge  3$,
               \item Condition (ii) implies that $\dim(D)\ge  4l\ge 4$ and $\dim(D_2)= \dim(\mathcal{L})$. 
          \end{enumerate}       
Next we adopt the definition of ascreen lightlike submanifolds used by Jin in \cite{ji} for the case of contact ambient manifold in case of lightlike submanifolds of an almost contact manifold.
\begin{definition}\label{def3} \cite{ji} {\rm
A lightlike submanifold $M$, immersed in a semi-Riemannian manifold $\overline{M}$ is said to be ascreen  if the structural vector field, $\xi$, belongs to $\mathrm{Rad}\, TM$ $\oplus l\mathrm{tr}(T M)$.
}
\end{definition}
Note that, since $\mathcal{L}$ defined in Definition \ref{def2} is a subbundle of $l\mathrm{tr}(T M)$, there is a complementary subbundle $\nu$ of $l\mathrm{tr}(T M)$ such that 
$$
l\mathrm{tr}(T M) = \mathcal{L}\perp\nu.
$$
It is easy to check that the complementary subbundle $\nu$ is invariant under $\overline{\phi}$, i.e. $\overline{\phi}\nu =\nu$.

Let $M$ be an ascreen QGCR-lightlike submanifold of an indefinite nearly Sasakian manifold $\overline{M}$. Then by Definition \ref{def3}, the structural vector field $\xi\in\mathrm{Rad}\, TM$ $\oplus l\mathrm{tr}(T M)$. This means that $\xi$ is either in $\mathrm{Rad}\,TM$ or $l\mathrm{tr}(T M)$. If $\xi\in \mathrm{Rad}\, T M$, then $\xi \in D_{2}$ since $\overline{\phi}D_{1} = D_{1}$ and $\overline{\phi}\xi=0$. On the other hand, if $\xi\in l\mathrm{tr}(T M)$, then  $\xi\in\Gamma(\mathcal{L})$ because of the fact that $\overline{\phi}\nu =\nu$ and $\overline{\phi}\xi=0$. Therefore, we have the following.
\begin{lemma}\label{Lemma1}
If $M$ is an ascreen QGCR-lightlike submanifold of an indefinite nearly Sasakian manifold $\overline{M}$, then $\xi\in\Gamma(D_{2}\oplus\mathcal{L})$.
\end{lemma}

\begin{theorem}\label{asc}
 Let $(M,g,S(TM),S(TM^\perp))$ be a 3-lightlike QGCR submanifold of an indefinite almost contact manifold $(\overline{M}, \overline{g}, \overline{\phi},\xi,\eta)$. Then $M$ is ascreen lightlike submanifold if and only if $\overline{\phi}\mathcal{L}=\overline{\phi} D_{2}$.
\end{theorem}
 \begin{proof} 
 Suppose that $M$ is ascreen. Then, by Lemma \ref{Lemma1}, $\xi\in\Gamma(D_{2}\oplus\mathcal{L})$. Since $M$ is a 3-lightlike QGCR submanifold, and $\mathrm{Rad}\, T M= D_{1}\oplus D_{2}$ with $\overline{\phi}D_{1}= D_{1}$ and $l\mathrm{tr}(T M) = \mathcal{L}\perp\nu$ with $\overline{\phi}\nu=\nu$, the distributions $D_{2}$ and $\mathcal{L}$ are of rank 1. Consequently,  
  \begin{equation}\label{s80}
   \xi=a E+ b N, 
  \end{equation}
  where $E\in\Gamma(D_2)$ and $N\in\Gamma(\mathcal{L})$, and $a=\eta(N)$ and $b=\eta(E)$ are non-zero smooth functions. Applying $\overline{\phi}$ to the first relation of (\ref{s80}) and using the fact that $\overline{\phi}\xi=0$, we get
\begin{equation}\label{s82}
   a \overline{\phi}E+b \overline{\phi}N=0.
  \end{equation}
From (\ref{s82}), one gets $\overline{\phi}E=\omega\overline{\phi}N$, where $\omega=-\frac{b}{a}\neq0$, a non vanishing smooth function. This implies that $\overline{\phi}\mathcal{L}\cap\overline{\phi} D_2\neq\{0\}$.   Since $rank(\overline{\phi} D_2)=rank(\overline{\phi}\mathcal{L})=1$, it follows that $\overline{\phi}\mathcal{L}=\overline{\phi} D_2$.

Conversely, suppose that $\overline{\phi}\mathcal{L}=\overline{\phi} D_2$. Then, there exists a non-vanishing smooth function $\omega$ such that 
\begin{equation}\label{s84}
 \overline{\phi}E=\omega\overline{\phi}N. 
\end{equation}
Taking the $ \overline{g}$--product of (\ref{s84}) with respect to $\overline{\phi}E$ and $\overline{\phi}N$ in turn, we get
\begin{equation}\label{s85}
 b^2=\omega(ab-1)\;\;\mbox{and}\;\;\omega a^2=a b-1.
\end{equation}
Since $\omega\neq 0$, by (\ref{s85}), we have $a\neq0$, $b\neq0$ and $b^2=(\omega a)^2$. The latter gives $b= \pm \omega a$. The case $b=\omega a $ implies that $a b =\omega a ^2=a b -1$, which is a contradiction. Thus $b =-\omega a $, from which $2a b=1$. Since $\omega =-\frac{b}{a}$, $a\neq0$ and $\overline{\phi}E=\omega\overline{\phi}N$, it is easy to see that $a \overline{\phi}E+b \overline{\phi}N=0$. Applying $\overline{\phi}$ to this equation, and using the first relation in (\ref{equa1}), together with $2ab=1$, we get 
$\xi=aE+bN$. Therefore, $M$ is ascreen lightlike submanifold of $\overline{M}$.
 \end{proof} 
In the ascreen QGCR-lightlike submanifold case, the item (ii) of Definition \ref{def2} implies that $\dim(D)\ge 4l\ge 4$ and $\dim(D_2)=\dim(\mathcal{L})$. Thus $\dim(M)\ge 7$ and $\dim(\overline{M})\ge 11$, and any 7-dimensional ascreen QGCR-lightlike submanifold is 3-lightlike.

As an example for QGCR-lightlike submanifold of indefinite nearly Sasakian manifold, we have the following. 
\begin{example}\label{exa1}
{\rm
Let $\overline{M}=(\mathbb{R}_4^{11}, \overline{g})$ be a semi-Euclidean space,  where $\overline{g}$ is of signature $(-,-,+,+,+, - ,  -,+,+,+,+)$ with respect to the canonical basis
\begin{equation*}
 (\partial x_1,\partial x_2,\partial x_3,\partial x_4,\partial x_5,\partial y_1,\partial y_2,\partial y_3,\partial y_4,\partial y_{5},\partial z).
\end{equation*}
Let $(M,g)$ be a submanifold of $\overline{M}$ given by 
\begin{equation*}
 x^1=y^4,\;\; y^1=-x^4,\;\; z= x^2\sin\theta + y^2\cos\theta \;\;\mbox{and}\;\; y^5=(x^5)^{\frac{1}{2}},
 \end{equation*}
 where $\theta\in(0,\frac{\pi}{2})$. By direct calculations, we can easily check that the vector fields
\begin{align*} 
 E_1 & =\partial x_4+\partial y_1+y^4\partial z,\;\;\; E_2=\partial x_1-\partial y_4+y^1\partial z,\\
 E_3 & =\sin\theta \partial x_2 +\cos\theta \partial y_2+\partial z, \; \;\; X_1=2y^5\partial x_5+\partial y_5+2(y^5)^2\partial z,\\ 
 X_2 & =-\cos\theta \partial x_2 +\sin\theta \partial y_2-y^2\cos\theta\partial z,\;\;\;  X_3=2\partial y_3,\; X_4=2(\partial x_3+y^3\partial z), 
\end{align*}
form a local frame of $TM$. From this, we can see that $\mathrm{Rad} \, TM$ is spanned by $\{E_1, E_2, E_3\}$, and therefore, $M$ is 3-lightlike. Further, $\overline{\phi}_0 E_1=E_2$, therefore we set $D_1=\mbox{Span}\{E_1,E_2\}$. Also $\overline{\phi}_0 E_3=-X_2$ and thus $D_2=\mathrm{Span}\{E_3\}$. It is easy to see that $\overline{\phi}_0 X_3=X_4$, so we set $D_0=\mathrm{Span}\{X_3,X_4\}$. On the other hand, following direct calculations, we have 
\begin{align} 
 N_1 & =\frac{1}{2}(\partial x_4-\partial y_1+y^4\partial z),\;\; \;  N_2=\frac{1}{2}(-\partial x_1-\partial y_4+y^1\partial z),\nonumber\\
 N_3 & =\frac{1}{2}(-\sin\theta \partial x_2 -\cos\theta \partial y_2+\partial z),\;\; \;  W=\partial x_5-2y^5\partial y_5+y^5\partial z, \nonumber
\end{align}
from which $l\mathrm{tr}(TM)=\mathrm{Span}\{N_1,N_2,N_3\}$ and $S(TM^\perp)=\mathrm{Span}\{W\}$. Clearly, $\overline{\phi}_0 N_2=-N_1$. Further, $\overline{\phi}_0 N_3=\frac{1}{2} X_2$ and thus $\mathcal{L}=\mbox{Span}\{N_3\}$. Notice that $\overline{\phi}_0N_3=-\frac{1}{2}\overline{\phi}_0 E_3$ and therefore $\overline{\phi}_0\mathcal{L}=\overline{\phi}_0 D_2$. Also, $\overline{\phi}_0 W=-X_1$ and therefore $\mathcal{S}=\mbox{Span}\{W\}$. Finally, we calculate $\xi$ as follows; Using Theorem \ref{asc} we have $\xi=a E_3+b N_3$. Applying $\overline{\phi}_0$ to this equation we obtain $a\overline{\phi}_0 E_3+b\overline{\phi}_0 N_3=0$. Now, substituting for $\overline{\phi}_0 E_3$ and  $\overline{\phi}_0 N_3$ in this equation we get $2a=b$, from which we get $\xi=\frac{1}{2}(E_3+2N_3)$. Since $\overline{\phi}_0\xi=0$ and $\overline{g}(\xi,\xi)=1$, we conclude that $(M,g)$ is an ascreen QGCR-lightlike submanifold of $\overline{M}$. }
\end{example}
\begin{proposition}
There exist no co-isotropic, isotropic or totally lightlike proper QGCR-lightlike submanifolds of an indefinite nearly Sasakian manifold.
\end{proposition}

\section{Some characterization theorems}\label{existence}

 In this section, we discuss an existence and some non-existence theorems for proper QGCR-lightlike submanifolds of an indefinite nearly Sasakian manifold $(\overline{M},\overline{\phi},\eta,\xi,\overline{g})$.
\begin{theorem}\label{Thm1}
   There exist no totally umbilical  proper QGCR-lightlike submanifolds $(M,g,S(TM)$, $S(TM^\perp))$ of an indefinite nearly Sasakian manifold $(\overline{M},\overline{\phi}$ $,\eta,\xi,\overline{g})$ with the structure vector field $\xi$ tangent to $M$.
\end{theorem}
 \begin{proof}
  Suppose that $\xi\in\Gamma(TM)$ and that $M$ is totally umbilical in $\overline{M}$. Then, $\xi=\xi_R+\xi_S$ and $b_i=c_\alpha=0$. Using (\ref{v10}) and (\ref{eq11}), we get
         \begin{equation}\label{eq15} 
  -\overline{\phi} X=\overline{H}X+\nabla_X \xi+\sum_{i=1}^rh_i^l(X,\xi)N_i+\sum_{\alpha=r+1}^nh_\alpha^s(X,\xi)W_\alpha,
          \end{equation}        
 for all $X\in \Gamma(TM)$. Taking the $\overline{g}$--product of (\ref{eq15}) with respect to $W_\alpha\in \Gamma(\mathcal{S})$  we  get
  \begin{equation}\label{eqw}
   g(X,\overline{\phi} W_\alpha)=\overline{g}(\overline{H}X, W_\alpha)+ \epsilon_\alpha h_\alpha^s(X,\xi),\;\;\forall\,X\in \Gamma(TM).
  \end{equation}
Now, letting  $X=\overline{\phi} W_\alpha$ in (\ref{eqw}) we obtain
  \begin{equation}\label{v2}
  g(\overline{\phi} W_\alpha,\overline{\phi} W_\alpha)=\overline{g}(\overline{H}\,\overline{\phi} W_\alpha, W_\alpha)+ \epsilon_\alpha h_\alpha^s(\overline{\phi} W_\alpha,\xi).
  \end{equation}
Since $c_\alpha=\epsilon_\alpha\eta(W_\alpha)=0$, then $-\overline{H}\,\overline{\phi} W_\alpha=(\overline{\nabla}_{W_\alpha} \overline{\phi})\xi+W_\alpha$ and the first term on the right hand side of (\ref{v2}) therefore simplifies as follows using (\ref{eqz})
 \begin{align}
  -\overline{g}(\overline{H}\,\overline{\phi} W_\alpha, W_\alpha)&= \overline{g}((\overline{\nabla}_{W_\alpha} \overline{\phi})\xi, W_\alpha)+\overline{g}(W_\alpha, W_\alpha) \nonumber\\
 &= -\overline{g}(\xi, (\overline{\nabla}_{W_\alpha} \overline{\phi})W_\alpha)+\overline{g}(W_\alpha, W_\alpha)\nonumber\\
 &= -\overline{g}(\xi,\bar{g}(W_\alpha, W_\alpha)\xi) +\overline{g}(W_\alpha, W_\alpha) \nonumber\\
 &= -\overline{g}(W_\alpha, W_\alpha)+\overline{g} (W_\alpha, W_\alpha)=0.
\end{align}
Then substituting $ \overline{g}(\overline{H}\,\overline{\phi} W_\alpha, W_\alpha)=0$ in (\ref{v2}) we obtain
  \begin{equation}\label{v16}
  g(\overline{\phi} W_\alpha,\overline{\phi} W_\alpha)= \epsilon_\alpha h_\alpha^s(\overline{\phi} W_\alpha,\xi).
  \end{equation}
By virtue of the fact that $M$ is totally umbilical in $\overline{M}$, (\ref{v16}) yields
 \begin{equation}\label{v17}
   g(\overline{\phi} W_\alpha,\overline{\phi} W_\alpha)=\epsilon_\alpha\mathscr{H}_\alpha^s  g(\overline{\phi} W_\alpha,\xi)=0.
  \end{equation}
Then, simplifying (\ref{v17}) while considering $\eta(W_\alpha)=0$, we get $g(\overline{\phi}W_\alpha, \overline{\phi}W_\alpha)=g(W_\alpha, W_\alpha)=\epsilon_\alpha=0$, which is a contradiction. 
\end{proof}
We notice from the above theorem that if $\xi$ is tangent to $M$, then $\bar{g}(\overline{H}\,\overline{\phi} W_\alpha, W_\alpha)$ $=0$. It is easy to see that $\bar{g}(\overline{H}X, W_\alpha)=0$, for all $X\in\Gamma(\overline{\phi}\mathcal{S})$. Hence, $\overline{H}X$ has no component along $\mathcal{S}$ for all $X\in\Gamma(\overline{\phi}\mathcal{S})$.
\begin{corollary}\label{T5}
   There exist no totally geodesic proper QGCR-lightlike submanifolds $(M,g,S(TM),S(TM^\perp))$ of an indefinite nearly Sasakian manifold $(\overline{M},\overline{\phi},$ $\eta,\xi,\overline{g})$ with the structure vector field $\xi$ tangent to $M$.
\end{corollary}
 Using Theorem \ref{Thm1} and Corollary \ref{T5} above we get the following non-existence theorem;
\begin{theorem}
   There exist no totally umbilical or totally geodesic proper QGCR-lightlike submanifolds $(M,g,S(TM)$ $,S(TM^\perp))$ of an indefinite nearly Sasakian manifold $(\overline{M},\overline{\phi},\eta,\xi,\overline{g})$ with the structure vector field $\xi$ tangent to $M$.
\end{theorem}
When the structure vector field $\xi$ is normal, we have the following.
\begin{theorem}
   There exist no proper QGCR-lightlike submanifolds $(M,g,S(TM)$, $S(TM^\perp))$ of an indefinite nearly Sasakian manifold $(\overline{M},\overline{\phi},\eta,\xi,\overline{g})$ with the structure vector field $\xi$ normal to $M$.
\end{theorem}
\begin{proof}
Suppose that  $\xi\in \Gamma(TM^\perp)$, then 
\begin{equation}\label{eqb}
 \xi=\xi_R+\xi_{S^\perp},\;\;\xi_l=\xi_S=0, \;\; b_i=0, \;\; a_i\neq 0\hspace{.2cm}\mbox{and} \;\; c_{\alpha} \neq 0.
\end{equation}
Differentiating the first equation of (\ref{eqb}) with respect to $X$ and using (\ref{eq10}), (\ref{eq11}) and (\ref{eq32}), we get 
   \begin{align}\label{a} 
       -\overline{\phi} X &= \sum_{i=1}^rX(a_i)E_i+\sum_{\alpha=r+1}^nX(c_\alpha)W_\alpha\nonumber\\
              &+\sum_{i=1}^ra_i \{\nabla_X E_i +\sum_{j=1}^r h_j^l(X,E_i)N_j+\sum_{\beta=r+1}^n h_\beta^s(X,E_i)W_\beta\}\nonumber\\
              &+\sum_{\alpha=r+1}^nc_\alpha\{-A_{W_\alpha} X+\sum_{i=1}^r \varphi_{\alpha i}(X) N_i +\sum_{\beta=r+1}^n \sigma_{\alpha\beta}(X)W_\beta\}+\overline{H}X, 
      \end{align}      
for all $X\in\Gamma(TM)$. Taking the $\overline{g}$-product of (\ref{a}) with respect to $E_k$ and $\overline{\phi} N_k\in \Gamma(S(TM))$ in turn, where $N_k\in\Gamma(\mathcal{L})$, we get 
   \begin{equation}\label{eqf}
       \overline{g}(X,\overline{\phi} E_k)=-\sum_{i=1}^r a_i h_i^l(X,E_k)-\sum_{\alpha=r+1}^nc_\alpha h_\alpha^s(X,E_k)+\overline{g}(\overline{H}X,E_k).
   \end{equation}
Replacing $X$ with $\overline{\phi} N_k$ in (\ref{eqf}) we obtain
   \begin{equation}\label{b}
       \overline{g}(N_k,E_k)=-\sum_{i=1}^r a_i h_i^l(\overline{\phi} N_k,E_k)-\sum_{\alpha=r+1}^nc_\alpha h_\alpha^s(\overline{\phi} N_k,E_k)+\overline{g}(\overline{H}\,\overline{\phi} N_k,E_k).
   \end{equation}
The $\overline{g}$-product with $\overline{\phi} N_k$ yields
   \begin{align}\label{c} 
        -\overline{g}(\overline{\phi} X, \overline{\phi}N_k)&= -\sum_{i=1}^ra_i\overline{g}(A_{E_i}^{*}X, \overline{\phi} N_k)+\sum_{i=1}^{r}a_{i} \sum_{j=1}^{r}h_{j}^{l}(X,E_i)\lambda_j(\overline{\phi} N_k)\nonumber\\
    &-\sum_{\alpha=r+1}^nc_\alpha \overline{g}(A_{W_\alpha}X, \overline{\phi} N_k)+\sum_{\alpha=r+1}^nc_\alpha \sum_{j=1}^r\varphi_{\alpha j}(X)\lambda_j(\overline{\phi} N_k)\nonumber\\
    &+\overline{g}(\overline{H}X, \overline{\phi} N_k).
    \end{align}
   Now, using (\ref{eqc}), (\ref{eqe}) and (\ref{eq07}) in (\ref{c}), we obtain
   \begin{align*} 
       \overline{g}(\overline{\phi} X,\overline{\phi} N_k)& = \sum_{i=1}^ra_i\overline{g}(A_{E_i}^*X, \overline{\phi} N_k)+\sum_{\alpha=r+1}^nc_\alpha \overline{g}(A_{W_\alpha}X, \overline{\phi} N_k)\\
                                     &-\overline{g}(\overline{H}X,\overline{\phi} N_k),
   \end{align*}
which on replacing $X$ with $E_k$  and simplifying gives
    \begin{equation}\label{eqr}
  \begin{aligned}
     \overline{g}(E_k,N_k)=&\;b_ka_k+\sum_{i=1}^ra_ih_i^l(E_k,\overline{\phi} N_k)+\sum_{\alpha=r+1}^nc_\alpha h_\alpha^s(E_k, \overline{\phi}N_k)\\
       &-\overline{g}(\overline{H}E_k,\overline{\phi} N_k).
     \end{aligned}
    \end{equation}
Adding (\ref{b}) to (\ref{eqr}) yields
    \begin{equation}\label{g1}
      2\overline{g}(E_k,N_k)=\overline{g}(\overline{H}\,\overline{\phi}N_k,E_k)-\overline{g}(\overline{H}E_k,\overline{\phi} N_k).
    \end{equation}
But $\overline{H}$ is skew-symmetric and thus (\ref{g1}) becomes
   \begin{equation}\label{g2}
      \overline{g}(E_k,N_k)=\overline{g}(\overline{H}\,\overline{\phi} N_k,E_k)=1.
    \end{equation}
By virtue of (\ref{g2}), it easy to see that $\overline{H}\,\overline{\phi} N_k\in\Gamma(l\mathrm{tr}(TM))$, particularly, in the direction of $N_k$. Hence, there exist a non vanishing smooth function $f_k$ such that $\overline{H}\,\overline{\phi} N_k=f_{k} N_k$. Taking the $\overline{g}$-product of this equation with respect to $\xi$, we get $0=\overline{g}(\overline{H}\,\overline{\phi} N_k,\xi)=\overline{g}(f_k N_k,\xi)=f_k\overline{g}( N_k,\xi)=f_k a_k$, from which $a_k=0$, a contradiction. Therefore, in a proper QGCR-lightlike submanifold of an indefinite nearly Sasakian manifold, $\xi$ does not belong to $TM^\perp$.
\end{proof}
 In particular, we have the following.
 \begin{corollary}
   There  exist no totally umbilical or totally geodesic proper QGCR-lightlike submanifolds $(M, g, S(TM),  S(TM^\perp))$ of an indefinite nearly Sasakian manifold $(\overline{M},\overline{\phi},\eta,\xi,\overline{g})$ with the structure vector field $\xi$ normal to $M$.
 \end{corollary}

\begin{corollary}
  Let $(M,g,S(TM),S(TM^\perp))$ be a proper  QGCR-lightlike submanifold of an indefinite nearly Sasakian manifold $(\overline{M},\overline{\phi},\eta,\xi,\overline{g})$. If the structure vector field $\xi$ is normal to $M$, then
 \begin{enumerate}
  \item $ \overline{H}X$ belongs to $l\mathrm{tr}(TM)$ for all $X\in\Gamma(\overline{\phi}\mathcal{L})$.  
  \item $ \overline{H}X$ belongs to $\mathrm{Rad}\,TM$ for all $X\in\Gamma(\overline{\phi}D_2)$.
 \end{enumerate}
\end{corollary}
\begin{theorem}
   There exist no totally umbilical proper QGCR-lightlike submanifolds $(M,g,S(TM), S(TM^\perp))$, with totally umbilical screen distributions, of an indefinite nearly Sasakian manifold $(\overline{M},\overline{\phi},\eta,\xi,\overline{g})$ with the structure vector field $\xi$ transversal to $M$.
\end{theorem}
\begin{proof}
Suppose that $\xi\in \Gamma(\mathrm{tr}(TM))$ and that $M$ is totally umbilical in $\overline{M}$, then 
\begin{equation}\label{eqA}
 \xi=\xi_l+\xi_{S^\perp},\;\;\xi_R=\xi_S=0, \;\; a_i=0, \;\; b_i\neq 0\;\;\mbox{and} \;\; c_\alpha \neq 0.
\end{equation}
 Differentiating the first equation of  (\ref{eqA}) with respect to $X$ and using (\ref{eq10}), (\ref{eq31}) and (\ref{eq32}), we get 
 \begin{align*}
  -\overline{\phi} X & = \sum_{i=1}^rX(b_i)N_i+\sum_{\alpha=r+1}^nX(c_\alpha)W_\alpha+\sum_{i=1}^rb_i \{-A_{N_i} X+\sum_{j=1}^r \tau_{ij}(X) N_j\\
         &+\sum_{\alpha=r+1}^n \rho_{i\alpha}(X)W_\alpha\}+\sum_{\alpha=r+1}^nc_\alpha\{-A_{W_\alpha} X+\sum_{i=1}^r \varphi_{\alpha i}(X) N_i\\
         &+\sum_{\beta=r+1}^n \sigma_{\alpha\beta}(X)W_\beta\}+\overline{H}X,
 \end{align*}
for all $X\in\Gamma(TM)$. Now, taking the $\overline{g}$-product of the above equation with respect to $\overline{\phi} N_k\in \Gamma(S(TM))$ where $N_k\in\Gamma(\mathcal{L})$, we get
    \begin{align}\label{eqB} 
     -\overline{g}(\overline{\phi} X, \overline{\phi}N_k)&= -\sum_{i=1}^rb_i g(A_{N_i}X, \overline{\phi}N_k)-\sum_{\alpha=r+1}^nc_\alpha g(A_{W_\alpha}X, \overline{\phi}N_k)\nonumber\\
      &+\overline{g}(\overline{H}X,\overline{\phi} N_k). 
    \end{align}
Replacing $X$ with $E_k\in\Gamma(D_2)$ in (\ref{eqB}), we obtain    
    \begin{align}\label{j} 
     -\overline{g}(\overline{\phi} E_k, \overline{\phi} N_k)& = -\sum_{i=1}^rb_i g(A_{N_i}E_k,\overline{\phi} N_k)-\sum_{\alpha=r+1}^nc_\alpha g(A_{W_\alpha}E_k,\overline{\phi} N_k)\nonumber\\
     &+\overline{g}(\overline{H}E_k,\overline{\phi} N_k). 
    \end{align}
Substituting (\ref{eqe}) and the first equation of (\ref{k}) in (\ref{j}) gives
      \begin{align}\label{i} 
     -\overline{g}(\overline{\phi} E_k, \overline{\phi} N_k)&= -\sum_{i=1}^rb_i h_i^*(E_k,\overline{\phi} N_k)-\sum_{\alpha=r+1}^nc_\alpha h_\alpha^s(E_k,\overline{\phi} N_k)\nonumber\\
         &+\bar{g}(\overline{H}E_k,\overline{\phi} N_k).
    \end{align}  
Since $M$ is totally umbilical in $\overline{M}$, with a totally umbilical screen, then (\ref{i}) yields
\begin{equation}\label{td}
  -\overline{g}( \overline{\phi}E_k, \overline{\phi} N_k)=\overline{g}(\overline{H}E_k, \overline{\phi}N_k),
\end{equation}
which reduces to $\overline{g}(E_k, N_k)=\overline{g}(\overline{\phi}\,\overline{H}E_k, N_k)=1$.  It is easy to see from this equation that $\overline{\phi}\,\overline{H}E_k\in\Gamma(\mathrm{Rad}\,TM)$. In particular, there exist non vanishing smooth functions $w_{k}$ such that $\overline{\phi}\,\overline{H}E_k=w_k E_k$. Taking the $\overline{g}$--product of this lat equation with respect to $\xi$, we obtain $0=\overline{g}(\overline{\phi}\,\overline{H}E_k,\xi)=w_k\overline{g}(E_k,\xi)=w_{k} b_{k}$. Hence, $b_{k}=0$, and this contradiction completes the proof.
\end{proof}
 \begin{corollary} 
    There exist no totally geodesic proper  QGCR-lightlike submanifolds $(M,g,S(TM),S(TM^\perp))$, with totally geodesic screen distributions, of an indefinite nearly Sasakian manifold $(\overline{M},\overline{\phi},\eta,\xi,\overline{g})$ with the structure vector field $\xi$ transversal to $M$.
 \end{corollary}

 Next, we consider the special case $\overline{H}=0$. In particular, the indefinite nearly Sasakian manifold $(\overline{M},\overline{\phi},\eta,\xi,\overline{g})$ with $\overline{H}=0$ becomes Sasakian. An indefinite Sasakian manifold $\overline{M}$ is called an indefinite Sasakian space form, denoted by $\overline{M}(c)$, if it has the constant $\overline{\phi}$-sectional curvature $c$ \cite{ma1}. The curvature tensor $\overline{R}$ of the indefinite space form $\overline{M}(c)$ is given by 
\begin{align}\label{ms1}
       4\overline{R}(\overline{X},\overline{Y})\overline{Z} &= (c+3)\{\overline{g}(\overline{Y},\overline{Z})\overline{X}-\overline{g}(\overline{X},\overline{Z})\overline{Y}\}+(1-c)\{\eta(\overline{X})\eta(\overline{Z})\overline{Y}\nonumber\\
                      &-\eta(\overline{Y})\eta(\overline{Z})\overline{X}+\overline{g}(\overline{X},\overline{Z})\eta(\overline{Y})\xi-\overline{g}(\overline{Y},\overline{Z})\eta(\overline{X})\xi\nonumber\\
                      &+\overline{g}(\overline{\phi}\, \overline{Y},\overline{Z})\overline{\phi}\, \overline{X}+\overline{g}(\overline{\phi}\, \overline{Z},\overline{X})\overline{\phi}\, \overline{Y}-2\overline{g}(\overline{\phi}\, \overline{X},\overline{Y})\overline{\phi}\, \overline{Z}\},
\end{align}   
for any $\overline{X},\overline{Y},\overline{Z}\in \Gamma(T \overline{M})$.

Now, using (\ref{ms1}) we have the following existence theorem.
\begin{theorem}
 Let $(M,g,S(TM),S(TM^\perp))$ be a lightlike submanifold of an indefinite Sasakian space form $\overline{M}(c)$ with $c\neq1$. Then, $M$ is a QGCR-lightlike submanifold of $\overline{M}(c)$ if and only if
 \begin{enumerate}
  \item[(a)] The maximal invariant subspaces of $T_{p}M$, $p\in M$ define a distribution
 \begin{equation*}
      D=D_0\perp D_{1}, 
   \end{equation*}
 where $\mathrm{Rad}\, TM=D_{1}\oplus D_{2}$ and $D_{0}$ is a non-degenerate invariant distribution.
  \item[(b)] There exists a lightlike transversal vector bundle $l\mathrm{tr}(TM)$ such that
 \begin{equation*}
  \overline{g}(\overline{R}(X, Y)E, N) = 0,\;\forall\, X, Y \in\Gamma(D_{0}),\;E \in \Gamma(\mathrm{Rad}\,  TM),\;N\in\Gamma(l\mathrm{tr}(TM)).
 \end{equation*}
 \item[(c)] There exists a vector subbundle $M_{2}$ on $M$ such that 
\begin{equation*}
  \overline{g}(\overline{R}(X,Y)W, W^{\prime}) = 0,\;\;\forall\, W, W^{\prime} \in\Gamma(M_{2}),
 \end{equation*}
where $M_{2}$ is orthogonal to $D$ and $\overline{R}$ is the curvature tensor of $\overline{M}(c)$.
 \end{enumerate} 
\end{theorem}

\begin{proof}
 Suppose $M$ is a QGCR-lightlike submanifold of $\overline{M}(c)$ with $c\neq1$. Then, $D=D_0\perp D_{1}$ is a maximal invariant subspace. Next, from (\ref{ms1}), for $X,Y\in\Gamma(D_{0})$, $E\in\Gamma(D_{2})$ and $N\in\Gamma(l\mathrm{tr}(TM))$ we have 
 \begin{align*}
  \overline{g}(\overline{R}(X, Y)E, N)& =\frac{c-1}{4}\{\eta(X)\eta(E)\overline{g}(Y,N)-\eta(Y)\eta(E)\overline{g}(X,N)\\    &-2g(\overline{\phi}X,Y)\overline{g}(\overline{\phi}E,N)\}\\
  &=\frac{1-c}{2}g(\overline{\phi}X,Y)\overline{g}(\overline{\phi}E,N).
 \end{align*}
Since $g(\overline{\phi}X,Y)\neq0$ and $\overline{g}(\overline{\phi}E,N)=0$, we have $\overline{g}(\overline{R}(X, Y)E, N)=0$. Similarly, from (\ref{ms1}), one obtains
\begin{equation*}
 \overline{g}(\overline{R}(X,Y)W, W^{\prime})=\frac{1-c}{2}g(\overline{\phi}X,Y)\overline{g}(\overline{\phi}W,W^{\prime}),
\end{equation*}
 $\forall\,X,Y\in\Gamma(D_{0})$ and $W,W'\in\Gamma(\overline{\phi}\mathcal{S})$. Let $W=\overline{\phi}W_{1}$ and $W'=\overline{\phi}W_{2}$ with $W_{1},W_{2}\in\Gamma(\mathcal{S})$. Since $g(\overline{\phi}X,Y)\neq0$ and $\overline{g}(\overline{\phi}W,W')=\overline{g}(\overline{\phi}^{2}W_{1},\overline{\phi}W_{2})=\overline{g}(\overline{\phi}W_{1},W_{2})=0$. Therefore,  we have $\overline{g}(\overline{R}(X, Y)W, W^{\prime})=0$.
 
 Conversely, assume that (a), (b) and (c) are satisfied. Then (a) implies that $D = D_{0}\perp D_{1}$ is invariant. From (b) and (\ref{ms1}) we get 
 \begin{equation}\label{ms2}
  \overline{g}(\overline{\phi}E,N)=0,
 \end{equation}
which implies $\overline{\phi}E \in \Gamma(S(TM))$. Thus, some part of $\mathrm{Rad}\,TM$, say $D_{2}$, belongs to $S(TM)$ under the action of $\overline{\phi}$. Further, (\ref{ms2}) implies $\overline{g}(\overline{\phi}E,N)=\overline{g}(\overline{\phi}^{2}E,\overline{\phi}N)=\overline{g}(-E+\eta(E)\xi,\overline{\phi}N)=-\overline{g}(E,\overline{\phi}N)=0$. Therefore, a part of $l\mathrm{tr}(TM)$, say $\mathcal{L}$, also belongs to $S(TM)$ under the action of $\overline{\phi}$. On the other hand, (c) and (\ref{ms1}) imply $\overline{g}(\overline{\phi}W,W^{\prime})=0$. Hence we obtain $\overline{\phi}M_{2}\perp M_{2}$. Also, $g(\overline{\phi} E, W) = -g(E, \overline{\phi}W ) = -c_{\alpha}\eta(E)$ implies that generally $\overline{\phi}M_{2}\oplus \mathrm{Rad} TM$ or equivalently, $M_{2}\oplus \overline{\phi}\mathrm{Rad} TM$. Now, from $M_{2}\oplus \overline{\phi}\mathrm{Rad} TM$ and the fact that $\overline{\phi}D_{1}=D_{1}$, then $M_{2}\perp D_{1}$ and $M_{2}\oplus \overline{\phi}D_{2}$. This also tells us that $\overline{\phi}M_{2}$ has a component along $l\mathrm{tr}(TM)$, essentially coming from $\xi$. On the other hand, invariant and non-degenerate $D_{0}$ implies $g(\overline{\phi}W, X) = 0$, for $X\in\Gamma(D_{0})$. Thus, $M_{2} \perp D_{0}$ and $\overline{\phi}M_{2} \perp D_{0}$. Since $\xi\in\Gamma(T\overline{M})$, we sum up the above results and conclude that 
\begin{equation*}
 S(T M ) = \{\overline{\phi}D_{2} \oplus M_{1}\oplus M_{2}\}\perp D_{0},
\end{equation*}
where $M_{1}=\overline{\phi}\mathcal{L}$. Hence $M$ is QGCR-lightlike submanifold of $\overline{M}(c)$ and the  proof is completed.
\end{proof}
Note that conditions (b) and (c) are independent of the position of $\xi$ and hence valid for GCR-lightlike submanifolds \cite{ds3} and QGCR-lightlike submanifolds of an indefinte Sasakian space form $\overline{M}(c)$. When  $\xi$ is tangent to $M$, it is well known \cite{ca} that $\xi\in\Gamma(S(TM))$. In this case, one has a GCR-lightlike submanifold, in which $D_{2}\perp\overline{\phi}D_{2}$ is an invariant subbundle of $TM$, leading to $D=D_{1}\perp D_{2}\perp\overline{\phi}D_{2}\perp D_{0}$ as the maximal invariant subspace of $TM$. On the other hand, when $M$ is QGCR-lightlike submanifold then $\xi\in\Gamma(T\overline{M})$ and thus $D_{2}\perp\overline{\phi}D_{2}$ is generally not an invariant subbundle of $TM$ since the action of $\overline{\phi}$ on it gives a component along $\xi$. In particular, let $E\in\Gamma(D_{2})$ then $E+\overline{\phi}E\in\Gamma(D_{2}\perp\overline{\phi}D_{2})$. But on applying $\overline{\phi}$ to this subbundle and considering the fact that $\eta(E)\neq0$ we get $-E+\overline{\phi}E+\eta(E)\xi\notin \Gamma(D_{2}\perp\overline{\phi}D_{2})$. Hence, $D=D_{0}\perp D_{1}$ is the maximal invariant subbundle of $TM$. Further, in the case of QGCR-lightlike submanifold, $\overline{\phi}D_{2}\oplus M_{2}$. In fact, let $\overline{\phi}E\in\Gamma(\overline{\phi}D_{2})$ and $W=\overline{\phi}W_{1}\in\Gamma(M_{2})$, where $W_{1}\in\Gamma(\mathcal{S})$. Then, $\overline{g}(\overline{\phi}E,W)=\overline{g}(\overline{\phi}E,\overline{\phi}W_{1})=-\eta(E)\eta(W_{1})\neq0$. This explains the second direct sum in decomposition  $S(TM)=\{\overline{\phi}D_{2} \oplus M_{1}\oplus M_{2}\}\perp D_{0}$. For the case of GCR-lightlike submanifold, $\eta(E)=\eta(W_{1})=0$, hence $\overline{g}(\overline{\phi}E,W)=\overline{g}(\overline{\phi}E,\overline{\phi}W_{1})=0$. This implies that $\overline{\phi}D_{2}\perp M_{2}$ and hence the first direct orthogonal sum in the decomposition  $S(TM)=\{\overline{\phi}D_{2} \oplus M_{1}\}\perp M_{2}\perp D_{0}\perp \langle\xi\rangle$.

 \section{Integrability of the  distributions  $D$ and $\widehat{D}$}\label{Integra}
 
 Let $M$ be a QGCR-lightlike submanifold of an indefinite nearly Sasakian manifold $(\overline{M}, \overline{g}, \overline{\phi},\xi,\eta)$. From (\ref{eq05}), the tangent bundle of any QGCR lightlike submanifold, $TM$, can be rewritten as
   \begin{equation}\label{eq13}
       TM=D \oplus \widehat{D}, 
   \end{equation}
   where $D=D_0\perp D_{1}$ and $\widehat{D}=\{D_{2}\perp\overline{\phi} D_{2}\} \oplus\overline{D}$.
   
Notice that $D$ is invariant with respect to $\overline{\phi}$ while $\widehat{D}$ is not generally anti-invariant with respect to $\overline{\phi}$. 
 
 Let $\pi$ and $\widehat{\pi}$  be the projections of $TM$ onto $D$ and $\widehat{D}$ respectively. Then, using the first equation of (\ref{eq13}) we can decompose $X$ as   
 \begin{equation}\label{16}
  X=\pi X + \widehat{\pi}X,\;\; \forall X\in \Gamma(TM).
 \end{equation}
 It is easy to see that $\overline{\phi}\pi X\in\Gamma(D)$. However, the action of $\overline{\phi}$ on $\widehat{\pi}X$ gives a tangential and transversal component due to a generalized $\xi$, i.e.,
 \begin{equation}\label{100}
  \overline{\phi} X=P_{1}X+P_{2}X+QX,\;\; \forall X\in \Gamma(TM), 
 \end{equation}
 where $P_{1}X=\overline{\phi}\pi X$ while  $P_{2}X$ is the tangential component of $\overline{\phi}\widehat{\pi}X$ and $QX$ is the transversal component of $\overline{\phi}X$, essentially coming from $\overline{\phi}\widehat{\pi}X$ since $\overline{\phi}D=D$. 
 
 By grouping the tangential and transversal parts in (\ref{100}), it is easy to see that 
  \begin{equation}\label{18}
  \overline{\phi} X=PX+QX,\;\; \forall X\in \Gamma(TM), 
 \end{equation}
 where $PX=P_{1}X+P_{2}X$. 
 
 Note that if  $X\in\Gamma{(D)}$, then $P_{2}X = QX=0$, and $ \overline{\phi} X=P_{1}X$. 
 
 The equation (\ref{18}) can be properly understood through the following specific case of vector field in $\overline{D}\subset\widehat{D}$. Let $\xi_{M}$ and $\xi_{\mathrm{tr} M}$ be the tangential and transversal components of $\xi$. If $X\in\Gamma(\overline{D})$ and since $ \overline{D}= \overline{\phi} \, \mathcal{S}\oplus \overline{\phi} \,\mathcal{L}$, then
 $$
 \overline{\phi} X = S X + L X -\{\eta(SX) + \eta(LY)\}\xi_{M} -\{\eta(SX) + \eta(LY)\}\xi_{\mathrm{tr} M}.
 $$ 
 Consequently, for $X\in\Gamma( \overline{D})$, 
  \begin{align}
    P_{1}X & =0,\nonumber\\
   P_{2}X & =-\{\eta(SX) + \eta(LY)\}\xi_{M},\nonumber\\
   \mbox{and}\;\; QX &=  S X + L X -\{\eta(SX) + \eta(LY)\}\xi_{\mathrm{tr} M}.\nonumber
  \end{align}
Similarly, for any $V\in\Gamma(\mathrm{tr}(TM))$, $V= SV +  LV$, and 
  \begin{equation}
  \overline{\phi} V = t V + f V, 
  \end{equation}
 where $tV$ and $fV$ are the tangential and transversal components of $\overline{\phi} V$, respectively.

   Differentiating (\ref{18}) with respect to $Y$ we get 
   \begin{equation}\label{20}
    \overline{\nabla}_YPX+\overline{\nabla}_YQX= \overline{\nabla}_Y\overline{\phi} X.
   \end{equation}
    Then using (\ref{eq11}),  (\ref{eq31}), (\ref{eq32}) and  (\ref{eqz}) we have
    \begin{equation}\label{21}
    \overline{\nabla}_YPX+\overline{\nabla}_YQX=\nabla_Y PX+h(PX,Y)-A_{QX}Y+\nabla_Y^t QX,
    \end{equation}
  and from (\ref{eqz}), we have;
  \begin{align}\label{22} 
      \overline{\nabla}_Y\overline{\phi} X& =\overline{\phi}(\nabla_YX)+\overline{\phi}(\nabla_XY)+2\overline{\phi} h(X,Y)-\overline{\nabla}_X\overline{\phi} Y\nonumber\\
   &+2\overline{g}(X,Y)\xi_M+2\overline{g}(X,Y)\xi_{\mathrm{tr} M}-\eta(Y)X-\eta(X)Y\nonumber\\
   &=P(\nabla_YX)+Q(\nabla_YX)+P(\nabla_XY)+Q(\nabla_XY)\nonumber\\
   &+2th(X,Y)+2fh(X,Y)-\nabla_XPY-\nabla_X^tQY\nonumber\\
   &-h(X,PY)+A_{QY}X+2\overline{g}(X,Y)\xi_M+2\overline{g}(X,Y)\xi_{\mathrm{tr} M}\nonumber\\
   &-\eta(Y)X-\eta(X)Y. 
     \end{align}
   Finally putting (\ref{21}) and (\ref{22}) in (\ref{20}) and then comparing the tangential and transversal components of the resulting equation, we obtain 
    \begin{align}\label{s102} 
         (\nabla_Y P)X+(\nabla_X P)Y&=  A_{QX}Y+A_{QY}X+2th(X,Y)\nonumber\\
    &+2\overline{g}(X,Y)\xi_M-\eta(X)Y-\eta(Y)X, 
     \end{align}
      and 
     \begin{align}\label{s103} 
        (\nabla_Y^TQ)X+(\nabla_X^TQ)Y& =  -h(PX,Y)-h(X,PY)\nonumber\\
                                    &+2fh(X,Y)+2\overline{g}(X,Y)\xi_{\mathrm{tr} M}, 
     \end{align}
  for all $X,Y\in\Gamma(TM)$, where
  \begin{equation}\label{25}
    (\nabla_Y P)X=\nabla_Y PX-P(\nabla_YX)\;\;\mbox{and}\;\;(\nabla_Y^T Q)X=\nabla_Y^t QX-Q(\nabla_YX).
  \end{equation} 
 \begin{proposition}
   Let $(M,g, S(TM),S(TM^\perp))$ be a QGCR-lightlike submanifold of an indefinite nearly Sasakian manifold $(\overline{M},\overline{\phi},\eta,\xi,\overline{g})$. Then,
    \begin{align}\label{23} 
         P[X,Y]=&\; -\nabla_YPX-\nabla_XPY+2P\nabla_{X}Y +A_{QX}Y+A_{QY}X\nonumber\\
               &+2th(X,Y) +2\overline{g}(X,Y)\xi_M-\eta(X)Y-\eta(Y)X, 
     \end{align}
     and 
     \begin{align}\label{26} 
         Q[X,Y]&=-\nabla_{Y}^tQX-\nabla_X^tQY+2Q\nabla_{X}Y -h(PX,Y)-h(X,PY)\nonumber\\
               &+2fh(X,Y) +2\overline{g}(X,Y)\xi_{\mathrm{tr} M}, 
     \end{align}
  for all $X,Y\in\Gamma(TM)$.
 \end{proposition}
 \begin{proof}
  The proof follows from (\ref{s102}) and (\ref{s103}).
 \end{proof}
  \begin{theorem}\label{th4}
  Let $(M,g, S(TM),S(TM^\perp))$ be a QGCR-lightlike submanifold of an indefinite nearly Sasakian manifold $(\overline{M},\overline{\phi},\eta,\xi,\overline{g})$. Then, the distribution $D$ is integrable if and only if 
  \begin{align}  
   & h(P_{1} X,Y) + h(X,P_{1} Y)=2(Q\nabla_XY+fh(X,Y)+\overline{g}(X,Y)\xi_{\mathrm{tr} M}), \nonumber\\
   \mbox{and}\;\; & P_{2}[X, Y]=0.   \nonumber
  \end{align}
  for all $X,Y\in \Gamma(D)$.
 \end{theorem} 
  \begin{proof} 
  The proof is a straightforward calculation.
 \end{proof} 
 
 The integrability of $\widehat{D}$ is discussed as follows. Note that the distribution $\widehat{D}$ is integrable if and only if, for any $X$, $Y\in\Gamma(\widehat{D})$, $[X, Y]\in\Gamma(\widehat{D})$. The latter is equivalent to $P_{1}[X, Y]=0$.
 \begin{theorem}\label{th5}
  Let $(M,g, S(TM),S(TM^\perp))$ be a QGCR-lightlike submanifold of an indefinite nearly Sasakian manifold $(\overline{M},\overline{\phi},\eta,\xi,\overline{g})$. Then, the distribution $\widehat{D}$ is integrable if and only if 
  \begin{align*}
 &A_{QX}Y+A_{QY}X-\nabla_YP_{2}X-\nabla_XP_{2}Y\\
 &+2(P_{1}(\nabla_{X}Y)+\overline{g}(X,Y)\xi_M+th(X,Y))\in\Gamma(\widehat{D}),
  \end{align*}
  for all $X,Y\in \Gamma(\widehat{D})$.
 \end{theorem}
  \begin{proof}
 Let  $X,Y\in \Gamma(\widehat{D})$, then it is easy to see that $P_{1}X=P_{1}Y=0$. Hence, $PX=P_{2}X$ and  $PY=P_{2}Y$. Now using (\ref{23}), we derive
   \begin{align}\label{s107}
   \overline{\phi}[X,Y]&=P[X,Y]+Q[X,Y]\nonumber\\
                       &=-\nabla_YPX-\nabla_XPY+2P\nabla_{X}Y+A_{QX}Y\nonumber\\
                       &+A_{QY}X+2th(X,Y) +2\overline{g}(X,Y)\xi_M-\eta(X)Y\nonumber\\
                       &-\eta(Y)X+Q[X,Y]\nonumber\\
                       &=-\nabla_YP_{2}X-\nabla_XP_{2}Y+2P_{1}\nabla_{X}Y+A_{QX}Y\nonumber\\
                       &+A_{QY}X+2th(X,Y) +2\overline{g}(X,Y)\xi_M+2P_{2}\nabla_{X}Y\nonumber\\
                       &-\eta(X)Y-\eta(Y)X+Q[X,Y].
  \end{align}
It is obvious from (\ref{s107}) that the last four terms belongs to $\widehat{D}$. Hence, the assertation follows from the remaining terms. 
 \end{proof} 
Let consider the lightlike submanifold $M$ given in Example \ref{exa1}. The distribution $D$ is spanned by $\{E_{1}, E_{2}, X_{3}, X_{4}\}$ while $\widehat{D}$ is spanned by $\{E_{3},\overline{\phi}_0 E_{3}, \overline{\phi}_0 W\}$. By straightforward calculations, we can see that $[E_{1},E_{2}]=2\partial z=2\xi$. Thus, $[E_{1},E_{2}]$ does not belong to $D$ and hence non integrable. On the other hand, $[E_3,\overline{\phi}_0 E_{3}]=-[E_3,X_{2}]=\cos^2\theta\partial z=\cos^2\theta\xi$. Since $\xi$ does not belong to $\widehat{D}$, we can see that $\widehat{D}$ is not integrable.

\section*{Acknowledgments}
S. Ssekajja is grateful to the African Institute of Mathematical Sciences, AIMS, in  S\'en\'egal and the University of KwaZulu-Natal, South Africa, for their financial support during this research. The authors thank the referees for helping them to improve the presentation.

\end{document}